\documentclass[11pt,article]{amsart}
\usepackage{amssymb}
\usepackage{amsfonts}
\usepackage{amsbsy}
\usepackage{latexsym}
\usepackage{centernot}
\usepackage{mathtools}
\usepackage{stmaryrd}

\usepackage{graphicx}

\usepackage{amsmath}

\newtheorem{theorem}{Theorem}[section]
\newtheorem{lemma}[theorem]{Lemma}
\newtheorem{proposition}[theorem]{Proposition}
\newtheorem{corollary}[theorem]{Corollary}

\theoremstyle{definition}

\newtheorem{example}[theorem]{Example}

\theoremstyle{remark}
\newtheorem{remark}[theorem]{Remark}

\numberwithin{equation}{section}

\newcommand{\la}{\left \langle}
\newcommand{\ra}{\right \rangle}
\newcommand{\lb}{\left\|}
\newcommand{\rb}{\right\|}

\newcommand\N{\mathbb{N}}

\newcommand\C{\mathbb{C}}
\newcommand\J{\mathbb{J}}

\newcommand\Z{\mathbb{Z}}

\newcommand\cH{\mathcal{H}}
\newcommand\cF{\mathcal{F}}

\newcommand{\spn}{\text{span}}
\newcommand{\sinc}{\text{sinc}}
\newcommand{\cspn}{\overline{\text{span}}}
\newcommand{\rank}{\text{rank}}

\newcommand{\ip}[2]{\left\langle\,#1 ,  #2 \, \right\rangle}

\begin{document}


\title{Signal Reconstruction from Frame and Sampling Erasures}

\author{David Larson}
\address{Department of Mathematics, Texas A\&M University, College Station, USA}
\email{larson@math.tamu.edu}

\author{Sam Scholze}
\address{Department of Mathematics, Texas A\&M University, College Station, USA}
\email{scholzes@math.tamu.edu}

\date{}



\begin{abstract}

We give some new methods for perfect reconstruction from frame and sampling erasures in finitely many steps.  By bridging an erasure set we mean replacing the erased Fourier coefficients of a function with respect to a frame by appropriate linear combinations of the non-erased coefficients.  We prove that if a minimal redundancy condition is satisfied bridging can always be done to make the reduced error operator nilpotent of index 2 using a bridge set of indices no larger than the cardinality of the erasure set.  This results in perfect reconstruction of the erased coefficients in one final matricial step.  We also obtain a new formula for the inverse of an invertible partial reconstruction operator. This leads to a second method of perfect reconstruction from frame and sampling erasures in finitely many steps. This gives an alternative to the bridging method for many (but not all) cases. The methods we use employ matrix techniques only of the order of
the cardinality of the erasure set, and are applicable to rather large finite erasure sets for infinite frames and sampling schemes as well as for finite frame theory. Some new classification theorems for frames are obtained and some new methods of measuring redundancy are introduced based on our bridging theory.

\end{abstract}

\maketitle


\bigskip


\section{Introduction}

Frame and sampling techniques are often used to analyze and digitize signals and images when they are represented as vectors or functions in a Hilbert space.  There is a large literature on the pure and applied mathematics of this subject (c.f. \cite{BBCE}, \cite{CKL}, \cite{Ch}, \cite{DS}, \cite{HL}).  A number of articles have been written on problems and methods for reconstruction from erasures (c.f. \cite{BP}, \cite{CK}, \cite{CK2}, \cite{GKK}, \cite{HP}).  We give some new methods for perfect reconstruction from frame and sampling erasures in finitely many steps.

Let $\{f_j\}$ be a Parseval frame for a Hilbert space $\mathcal{H}$, or more generally let $\{f_j,g_j\}$ be a dual pair of frames.  (See definitions below.)  Let $f$ be a vector in $\mathcal{H}$, and let $\Lambda$ be a finite subset of the index set. If $f$ is analyzed with $\{g_j\}$ and if the frame coefficients for $\Lambda$ are erased, then by bridging the erasures we mean replacing the erased coefficients with appropriate linear combinations of the non-erased coefficients. We show that bridging can always be done to make the resulting reduced error operator nilpotent of index 2 using a bridge set no larger than the cardinality of the erasure set.  From this, an algorithm for perfect reconstruction from erasures follows in one final simple step.  The resulting algorithms use only finite matrix methods of order the cardinality of the erasure set.
Frames can be infinite, such as Gabor and wavelet frames.  The only delimiter in a computational sense seems to be the size of the erasure set, which we take to be finite in this article.  This method adapts equally well to sampling theory, such as Shannon-Whittaker sampling theory (\cite{BF}, \cite{G}, \cite{Z}).  Our bridging results suggest some new classification techniques and new measures of redundancy for finite and infinite frames and sampling schemes.  We conclude with an apparently new formula for inverting the partial reconstruction operator.

We would like to thank Deguang Han for useful discussions on this work, and for piquing our interest in frame erasure problems in the recent interesting article \cite{PHM}.  We thank Stephen Rowe for useful Matlab and programming advice in the experimental phases of this work.  Many of our mathematical results were obtained after numerous computer experiments.

\section{Preliminaries}

A \textit{frame} $F$ for a Hilbert space $\cH$ is a sequence of vectors $\{f_j \} \subset \cH$ indexed by a finite or countable index set $\J$ for which there exist constants $0<A\leq B<\infty$ such that, for every $f \in \cH$,

\begin{equation}\label{frame}
 A\|f\|^2 \leq \sum_{j\in\J} |\ip{f}{f_j}|^2 \leq B\|f\|^2.
\end{equation}

The optimal constants are known as the upper and lower \textit{frame bounds}. A frame is called  \textit{tight} if $A=B$,  and is called a \textit{Parseval frame} if $A=B=1$. If we only require that a sequence $\{f_{j}\}$ satisfies the upper bound condition in (\ref{frame}), then $\{f_{j}\}$ is called a {\it Bessel sequence}. A frame which is a basis is called a Riesz basis. Orthonormal bases are special cases of Parseval frames. A Parseval frame $\{f_j\}$ for a Hilbert space $\cH$ is an orthonormal basis if and only if each $f_j$ is a unit vector.

The {\it analysis operator} $\Theta$  for a Bessel sequence $\{ f_j \}$ is a bounded linear operator from $\cH$ to $\ell^{2}(\mathbb{J})$ defined by
\begin{equation}\label{frameDef} \Theta f = \sum_{j \in \J}\ip{f}{f_{j}}e_{j}, \end{equation}
where $\{e_{j}\}$ is the standard orthonormal basis for $\ell^{2}(\mathbb{J})$. It is easily verified that
$$ \Theta^{*}e_{j} = f_{j}, \ \ \forall j \in \J. $$
The Hilbert space adjoint $\Theta^{*}$ is called the {\it synthesis operator} for $\{f_{j}\}$. The positive operator $S:=\Theta^{*}\Theta:\cH \rightarrow \cH$ is called the {\it frame operator}, or sometimes the {\it Bessel operator} if the Bessel sequence is not a frame, and we have
\begin{equation}\label{frameOp}
Sf = \sum_{j \in \J}\ip{f}{ f_{j}}f_{j},  \ \ \ \forall f \in \cH.
\end{equation}
We can also use rank one operator notation $(x \otimes y)(z) = \la z,y \ra x$ to write (\ref{frameOp}) as
$$ S = \sum_{j \in \J} f_j \otimes f_j. $$
Similarly, $\Theta = \sum_{j \in \J} e_j \otimes f_j$ and $\Theta^* = \sum_{j \in \J} e_j \otimes f_j$. The operator $\Theta \Theta^*: \ell^2(\J) \rightarrow \ell^2(\J)$ is called the Gramian operator (or Gram Matrix) and is denoted $Gr(F)$.  Then
$$Gr(F) = \sum_{j,k \in \J} \la f_k,f_j \ra e_j \otimes e_k = (\la f_k,f_j \ra)_{j, k}.$$

From (\ref{frameOp}) we obtain the {\it reconstruction formula (or frame decomposition)}
$$f = \sum_{j \in \J}\ip{f}{S^{-1}f_{j}}f_{j} = \sum_{j \in \J}\ip{f}{f_{j}}S^{-1}f_{j} \ \ \forall f \in \cH$$
where the convergence is in the norm of $\cH$.  The frame $\{S^{-1}f_{j}\}$ is called the {\it canonical or standard dual} of $\{f_{j} \}$. In the case that $\{f_{j}\}$ is a Parseval frame for $\cH$, we have $S = I$ and hence $ f = \sum_{j \in \J}\ip{f}{f_{j}}f_{j}, \ \ \forall f \in \cH. $ More generally, if a Bessel sequence $\{g_{j}\}$ satisfies a reconstruction formula
\[ f = \sum_{j \in \J}\ip{f}{g_{j}}f_{j} \ \ \forall f \in \cH \]
then $\{g_{j}\}$ is called an {\it alternate dual } of $\{f_{j} \}$.  (Hence $\{g_{j} \}$ is also necessarily a frame.) The canonical and alternate duals are often simply referred to as {\it duals}, and $(F,G):=\{f_{j}, g_{j} \}_{j \in \J}$ is called a {\it dual frame pair.}  The second frame $G$ in the ordered pair will be called the analysis frame and the first frame $F$ will be called the synthesis frame.

It will be convenient to define a {\it frame pair} which is not necessarily a dual frame pair  to be simply a pair of frames $F = \{ f_j \}$ and $G = \{ g_j \}$ indexed by the same set $\mathbb{J}$ for which the operator $\tilde{S}f = \sum\ip{f }{g_{j}}f_{j}$ is invertible. We will call the operator $\tilde{S}$ the {\it cross frame operator} for $F$ and $G$, and the operator $Gr(F,G) = \sum \ip{f_k}{g_j} e_j \otimes e_k$ the {\it cross Gramian}.  If $\{ f_1, \cdots, f_L \}$ and $\{ g_1, \cdots, g_L \}$ are finite sets of vectors, we will write $G(\{ f_1, \cdots, f_L \},\{ g_1, \cdots, g_L \})$ for the cross Gram matrix,
\begin{equation}
G(\{ f_1, \cdots, f_L \},\{ g_1, \cdots, g_L \}) = \left( \la f_k,g_j \ra \right)_{j,k} :=
\left( \begin{array}{cccc}
\la f_1,g_1 \ra & \la f_2,g_1 \ra & \cdots & \la f_L,g_1 \ra \\
\la f_1,g_2 \ra & \la f_2,g_2 \ra & \cdots & \la f_L,g_2 \ra \\
\vdots & \vdots & \ddots & \vdots \\
\la f_1,g_L \ra & \la f_2,g_L \ra & \cdots & \la f_L,g_L \ra
\end{array} \right).
\end{equation}
We will use this notation mainly when $\{ f_j \}$ and $\{ g_j \}$ are frames or subsets of frames.  It is useful to note that if $\{ f_1, \cdots, f_L \}$ and $\{ g_1, \cdots, g_L \}$ are two bases for the same Hilbert space $\cH$, then $G(\{ f_1, \cdots, f_L \},\{ g_1, \cdots, g_L \})$ is invertible.  Indeed, if $\{ e_j \}$ is an orthonormal basis for $\cH$, and $A$ and $B$ are invertible matrices with $Ae_j = f_j$ and $Be_j = g_j$, then $G(\{ f_1, \cdots, f_L \},\{ g_1, \cdots, g_L \})$ is just the matrix of $B^*A$ with respect to $\{ e_j \}$.

\section{Nilpotent Bridging}

Let $F = \{ f_j \}_{j \in \J}$ be a frame.  An {\it erasure set} for $F$ is defined to be simply a finite subset of $\J$.  We say that an erasure set $\Lambda$ for a frame $F$ satisfies the minimal redundancy condition if $\cspn\{ f_j : j \not \in \Lambda \} = \cH$.  The following elementary lemma is undoubtedly well-known.  Since it is important to this work we include a proof for completeness.  The steps of it also serve to elucidate some of the operator theoretic methods we use.

\begin{lemma}
Let $F = \{ f_j \}_{j \in \J}$ be a frame for $\cH$.  If $\Lambda \subset \J$ is an erasure set satisfying the minimal redundancy condition, then $\{ f_j : j \in \Lambda^c \}$ is a frame for $\cH$.
\end{lemma}
\begin{proof}
First consider the case where $F$ is a Parseval frame.  Assume that $\{ f_j : j \in \Lambda^c \}$ is not a frame for $\cH$.  Then the operator $R_\Lambda = \sum_{j \in \Lambda^c} f_j \otimes f_j$ for $\{ f_j : j \in \Lambda^c \}$ is not invertible.  Let $E_\Lambda = \sum_{j \in \Lambda} f_j \otimes f_j$.  Then, $R_\Lambda = I - E_\Lambda$, so $1 \in \sigma(E_\Lambda)$ (the spectrum of $E_\Lambda$).  Since $E_\Lambda$ is a finite rank operator, 1 is an eigenvalue, so there is a unit vector $x \in \cH$ so that $E_\Lambda x = x$.  Then, $R_\Lambda x = 0$.  Let $P = x \otimes x$ be the orthogonal projection onto $\C x$.  On one hand, we have
\[ P R_\Lambda P = P R_\Lambda (x \otimes x) = P ((R_\Lambda x) \otimes x ) = 0. \]
On the other, we have
\[ P R_\Lambda P = P \left( \sum_{j \in \Lambda^c} f_j \otimes f_j \right) P = \sum_{j \in \Lambda^c} P f_j \otimes P f_j. \]
But since each summand $P f_j \otimes P f_j$ is a positive operator, $Pf_j = 0$ for all $j \in \Lambda^c$.  That is, $x \perp \cspn\{ f_j : j \in \Lambda^c \}$. Therefore, $\Lambda$ does not satisfy the minimal redundancy condition with respect to $F$.  Hence the result holds for Parseval frames.

Now, consider the case where $F$ is a general frame.  Let $S = \sum_{j \in \J} f_j \otimes f_j$ be the frame operator for $F$.  Let $h_j = S^{-\frac{1}{2}} f_j$.  Then, $\{ h_j \}_{j \in \J}$ is a Parseval frame, and
\[ \cspn\{ h_j : j \in \Lambda^c \} = S^{-\frac{1}{2}} \cspn\{ f_j : j \in \Lambda^c \} = \cH. \]
So, $\Lambda$ satisfies the minimal redundancy condition with respect to $\{ h_j \}_{j \in \J}$.  Since the lemma holds for Parseval frames, $\sum_{j \in \Lambda^c} h_j \otimes h_j$ is invertible.  Thus,
\[ R_\Lambda = \sum_{j \in \Lambda^c} f_j \otimes f_j = \sum_{j \in \Lambda^c} S^{\frac{1}{2}} h_j \otimes S^{\frac{1}{2}} h_j = S^{\frac{1}{2}} \left( \sum_{j \in \Lambda^c} h_j \otimes h_j \right) S^{\frac{1}{2}} \]
is invertible.  Therefore, $\{ f_j : j \in \Lambda^c \}$ is a frame for $\cH$.
\end{proof}

Let $(F, G) = \{ f_j , g_j\}_{j \in \J}$ be a dual frame pair.  As above, an {\it erasure set} for $(F,G)$ is simply a finite subset of $\J$.  We say that an erasure set $\Lambda $ satisfies the {\it minimal redundancy condition for the dual frame pair} $(F,G)$ if $\cspn\{ g_j : j \not \in \Lambda \} = \cH$.  We point out that the minimal redundancy condition for a dual frame pair $(F,G)$ as we have defined it is a condition on only the analysis frame $G$.  The redundancy properties of the  synthesis frame $F$ play a role here only in that it is required to be a dual frame to $G$.  For the special case  where $G$ is the standard dual of $F$, $F$ and $G$ have the same linear redundancy properties.  The Parseval frame case, where $F = G$, is a special case of this.  For a dual pair $(F,G)$, if $\Lambda$ satisfies the minimal redundancy condition then since $\{g_j : j \in \Lambda^c\}$ is a frame for $\cH$ it has {\it some} frame dual (in general many duals) that will yield the reconstruction of $f$ from the
coefficients over
$\Lambda^c$, so there is enough information in $\{ \la f,g_j \ra : j \in \Lambda^c \}$ to reconstruct $f$.  On the other hand if $\Lambda$ fails the minimal redundancy condition then some nonzero vector $f$ will be orthogonal to $g_j$ for all $j \in \Lambda^c$, and hence no reconstruction of $f$ is possible using only the coefficients $\{\la f,g_j \ra : j \in \Lambda^c\}$.  This justifies the use of the word ``minimal'' in the description of the minimal redundancy condition.

Let $F$ be a {\it Parseval frame}.  If $\Lambda$ is an erasure set which satisfies the minimal redundancy condition, then $\{ f_j \}_{j \in \Lambda^c}$ still forms a frame, and the {\it partial reconstruction operator} $R_\Lambda := \sum_{j \in \Lambda^c} f_j \otimes f_j$ is the {\it frame operator} for the reduced frame $\{ f_j \}_{j \in \Lambda^c}$, hence it is invertible.  Let $f_R = R_\Lambda f$ be the partial reconstruction of the vector $f$.  It is possible to reconstruct $f$ from the ``good'' Fourier coefficients by $f = R_\Lambda^{-1} f_R$.
However, given a {\it dual frame pair} $(F,G)$ indexed by $\J = \{ 1,2, \cdots, N \}$ and an erasure set $\Lambda$ satisfying the minimal redundancy condition, the partial reconstruction operator $R_\Lambda := \sum_{j \in \Lambda^c} f_j \otimes g_j$ need not be invertible.  In fact invertibility of $R_\Lambda$ can fail even if both $F$ and $G$ separately satisfy the minimal redundancy condition for $\Lambda$.  The following simple example shows that this can happen and $R_\Lambda$ can even be the zero operator.

\begin{example}\label{noninv}
Let $\{ f_j,g_j \}_{j = 1}^N$ be a dual frame pair.  Suppose
\begin{eqnarray*}
f_j &=& f_{j+N} = f_{j+2N} \ \  1 \leq j \leq N \\
g_{j + N} &=& -g_j \ \  1 \leq j \leq N \\
g_{j+2N} &=& g_j \ \  1 \leq j \leq N.
\end{eqnarray*}
Then, it is easily verified that $\{ f_j,g_j \}_{j=1}^{3N}$ is a dual frame pair, and $\Lambda = \{ 1, 2, \cdots, N \}$ satisfies the minimal redundancy condition with respect to both frames.  However,
\[ R_\Lambda = \sum_{j = N+1}^{3N} f_j \otimes g_j = \sum_{j = N+1}^{2N} f_j \otimes g_j + \sum_{j = 2N + 1}^{3N} f_j \otimes g_j = \sum_{j = 1}^N f_j \otimes (-g_j) + \sum_{j = 1}^N f_j \otimes g_j = 0. \]
\qed
\end{example}

Even when $R_\Lambda$ is invertible, computing $R_\Lambda^{-1}$ can be a computationally costly process.  The error for the partial reconstruction is $f_E = f - f_R$, and the associated error operator for the partial reconstruction is $E_{\Lambda} = I - R_{\Lambda} = \sum_{j \in \Lambda}f_j \otimes g_j$.  Then $R_{\Lambda}^{-1} = (I - E_{\Lambda})^{-1}$, and if the norm, or more generally the spectral radius of $E_{\Lambda}$ is strictly less than $1$ then $R_{\Lambda}^{-1}$ can be computed using the Neumann series expansion $R_{\Lambda}^{-1} = I + E + E^2 + \cdots = \sum_{j = 0}^{\infty} E^j$.

For certain very special cases $(F,G)$, with corresponding erasure set $\Lambda$, the error operator, $E_\Lambda$ will be nilpotent of index 2, (i.e. $E_\Lambda^2 = 0$) such as the example below.  In this case, $R_\Lambda^{-1} = I + E_\Lambda$, and moreover, the error $f_E$ of $f$, can be obtained by applying the error operator to the partial reconstruction $f_R$ instead of $f$.  (That is, $f_E = E_\Lambda f = E_\Lambda (f_E + f_R) = E_\Lambda^2 f + E_\Lambda f_R = E_\Lambda f_R$.)

\begin{example}\label{nilpex}
Let $\{ e_1,e_2 \}$ be the standard orthonormal basis for $\C^2$.  Let \\ $F = \{ e_1,-e_1,e_1,e_2 \}$ and $G = \{ e_2,e_2,e_1,e_2 \}$.  Let $\Lambda = \{ 1 \}$.  Then $E_\Lambda = e_1 \otimes e_2$.  So
\[ E_\Lambda^2 = (e_1 \otimes e_2)(e_1 \otimes e_2) = \la e_1,e_2 \ra (e_1 \otimes e_2) = 0. \]
Therefore $R_\Lambda^{-1} = I + E_\Lambda$. \qed
\end{example}

If $\Omega$ is a subset of $\Lambda^c$ we can sometimes pre-condition $f_R$ to be a better first approximation to $f$ before applying an inversion operator which is a correspondingly pre-conditioned version of $R_{\Lambda}^{-1}$.
We will call such a pre-conditioning {\it bridging} if the method is to linearly replace each ``erased'' coefficient $\{c_j : j \in \Lambda \}$ with a weighted average of the coefficients in $\{c_j : j \in \Lambda \}^c$.  By ``linearly'' we mean that the same weights are used in the weighted average for each $f \in H$.  However, each index $j \in \Lambda$ can correspond to a different weighted average from the coefficients indexed by $\Lambda^c$.  The set of all the indices in $\Lambda^c$ that are used to bridge the coefficients with indices in $\Lambda$ is called the {\it bridge} for ${\Lambda}$.  We denote this by $\Omega$.

Let $\Lambda$ be an erasure set for a dual frame pair $(F,G)$ that satisfies the minimal redundancy condition with respect to $G$.  Since $\cspn\{ g_j : j \in \Lambda^c \} = \cH$, we could  pick indices $\Omega \subset \Lambda^c$ so that $\{ g_\omega: \omega \in \Omega \}$ is a basis for $\cH$.  For each $k \in \Lambda$, write
\[ g_k = \sum_{\omega \in \Omega} c_\omega^{(k)} g_\omega. \]
Then,
\[ \la f,g_k \ra = \la f,\sum_{\omega \in \Omega} c_\omega^{(k)} g_\omega \ra = \sum_{\omega \in \Omega} \overline{c_\omega^{(k)}} \la f,g_\omega \ra. \]
So, replacing the Fourier coefficient $\la f,g_k \ra$ with the weighted average of the $\la f,g_\omega \ra$ as above, we obtain a perfect reconstruction. But this method requires the cardinality $|\Omega|$ to be $\dim(\cH)$, which is far too large to be useful in a nontrivial reduction of the reconstruction process.  In this article we will consider only bridge sets $\Omega$ of the same (or lesser) cardinality as that of the erasure set $\Lambda$.

Let $(F,G)$ be a dual frame pair, $\Lambda$ be an erasure set, and $\Omega$ be our bridge set.  We will replace each Fourier coefficient $\la f,g_k \ra$ for $k \in \Lambda$ with a $\la f,g_k' \ra$ for some choice of $g_k' \in \spn\{ g_j : j \in \Omega \}$.  Then, our partial reconstruction with bridging is
\[ \tilde{f} = f_R + f_B \]
where $f_B = \sum_{j \in \Lambda} \la f,g_j ' \ra f_j$.  We call $f_B$ the bridging supplement and $B_\Lambda : = \sum_{j \in \Lambda} f_j \otimes g_j '$ the {\it bridging supplement operator}.  The reduced error is $f_{\tilde{E}} := f - \tilde{f}$, and the associated {\it reduced error operator} is $\tilde{E}_\Lambda = I - R_\Lambda - B_\Lambda$.  We have
\[ \tilde{E}_\Lambda f = f_{\tilde{E}} = \sum_{j \in \Lambda} \la f,g_j - g_j' \ra f_j. \]

There are various ways to choose the $g_k' \in \spn\{ g_j : j \in \Omega \}$, but in this paper, we choose $g_k'$ so that the reduced error operator is nilpotent of index 2.  Then the logic in the sentence just above Example \ref{nilpex} will apply, leading to perfect reconstruction in a final step.  It is straightforward to verify that the condition
\begin{equation}\label{perpcond}
f_j \perp (g_k - g_k ') \ \ \forall j,k \in \Lambda
\end{equation}
forces the reduced error operator to be nilpotent of index 2.  So, writing
\begin{equation}\label{coeffs}
g_k ' = \sum_{\ell \in \Omega} c_\ell^{(k)} g_ \ell
\end{equation}
we seek coefficients $c_\ell^{(k)}$ so that (\ref{perpcond}) is satisfied.  We have
\[ 0 = \la f_j, g_k - \sum_{\ell \in \Omega} c_\ell^{(k)} g_\ell \ra = \la f_j,g_k \ra - \sum_{\ell \in \Omega} \overline{c_\ell^{(k)}} \la f_j,g_\ell \ra. \]
For each $k \in \Lambda$, we obtain a system of $|\Lambda|$ equations with $|\Omega|$ unknowns:
\[ \la f_j,g_k \ra = \sum_{\ell \in \Omega} \overline{c_\ell^{(k)}} \la f_j,g_\ell \ra. \]
If we enumerate $\Lambda = \{ \lambda_j \}_{j=1}^L$ and $\Omega = \{ \omega_j \}_{j=1}^M$ we obtain the matrix equation

\begin{equation}\label{ME1}
 \left(
     \begin{array}{cccc}
       \la f_{\lambda_1},g_{\omega_1} \ra & \la f_{\lambda_1},g_{\omega_2} \ra & \cdots & \la f_{\lambda_1},g_{\omega_M} \ra \\
       \la f_{\lambda_2},g_{\omega_1} \ra & \la f_{\lambda_2},g_{\omega_2} \ra & \cdots & \la f_{\lambda_2},g_{\omega_M} \ra \\
       \vdots                     & \vdots                     & \ddots & \vdots \\
       \la f_{\lambda_L},g_{\omega_1} \ra & \la f_{\lambda_L},g_{\omega_2} \ra & \cdots & \la f_{\lambda_L},g_{\omega_M} \ra \\
     \end{array}
   \right)
\left(
  \begin{array}{c}
    \overline{c_{\omega_1}^{(k)}} \\
     \overline{c_{\omega_2}^{(k)}} \\
    \vdots \\
     \overline{c_{\omega_M}^{(k)}} \\
  \end{array}
\right)
=
\left(
  \begin{array}{c}
    \la f_{\lambda_1},g_k \ra \\
    \la f_{\lambda_2},g_k \ra \\
    \vdots \\
    \la f_{\lambda_L},g_k \ra \\
  \end{array}
\right)
\end{equation}
for all $k \in \Lambda$.  We call the matrix in (\ref{ME1}) the {\it bridge matrix} and denote it $B(F,G,\Lambda,\Omega)$.  Since the bridge matrix is independent of $k$, we can solve for all of the coefficients simultaneously with the equation

\begin{equation}\label{ME2}
\left( \la f_{\lambda_j},g_{\omega_k} \ra \right)_{j,k} \left( \overline{c_{\omega_j}^{\lambda_k}} \right)_{j,k} = \left( \la f_{\lambda_j},g_{\lambda_k} \ra \right)_{j,k}
\end{equation}
We can rewrite this equation as
\begin{equation}\label{BE}
B(F,G,\Lambda,\Omega) C = B(F,G,\Lambda,\Lambda)
\end{equation}
where $C$ denotes our coefficient matrix (actually, $C$ is the matrix of complex conjugates of the coefficients $c_{\omega_j}^{(\lambda_k)}$ in \ref{ME2}).
\begin{remark}
(1) The transpose of the bridge matrix $B(F,G,\Lambda,\Omega)$ is a skew (i.e. diagonal-disjoint) minor of the cross Gram matrix $G(F,G)$ of the frames $F$ and $G$, and the transpose of $B(F,G,\Lambda,\Lambda)$ is a principle minor of $G(F,G)$.
(2) The form of the bridge matrix in \ref{ME1} depends on the particular enumerations one takes of $\Lambda$ and $\Omega$.  However, for two different enumerations one bridge matrix will transform into the other by interchanging appropriate rows and columns, and so the norm and the rank of the matrices will be the same.  In particular, one will be invertible if and only if the other is.
\end{remark}

Given a dual frame pair $(F,G)$, and an erasure set $\Lambda$, a bridge set $\Omega$ is said to satisfy the {\it robust bridging condition} (or $\Omega$ is a {\it robust} bridge set) if equation (\ref{BE}) has a solution.

Now, given $f \in \cH$,
\[ f = f_{\tilde{E}} + \tilde{f}. \]
However, $\tilde{E}_\Lambda (f - \tilde{f}) = \tilde{E}_\Lambda ^ 2 f = 0$.  Thus, $f_{\tilde{E}} = \tilde{E}_\Lambda \tilde{f}$, and we can reconstruct $f$ from the good Fourier coefficients by
\begin{equation}
f = \tilde{f} + \tilde{E}_\Lambda \tilde{f}.
\end{equation}
Furthermore, $f_B \in \spn \{ f_j : j \in \Lambda \}$, so by $(\ref{perpcond})$, $\tilde{E}_\Lambda f_B = 0$.  Therefore, to reconstruct $f$, we have
\begin{equation}\label{reconstruction}
f = \tilde{f} + \tilde{E}_\Lambda f_R.
\end{equation}

Let $\alpha_j = \la f,g_j \ra$ and $\beta_j = \la f_R,g_j \ra$.  Then, $\alpha_j$ for $j \in \Omega$ are known coefficients, and the $\beta_j$ are computable.  The theorem below gives a direct algorithm for the reconstruction that involves nilpotent bridging and then applying the error operator.

\begin{theorem}\label{algorithm}
Let $(F,G)$ be a dual frame pair with erasure set $\Lambda$ satisfying the minimal redundancy condition, and $\Omega$ be a robust bridge set.  Assume $C = \left(\overline{c_{j}^{(k)}} \right)_{j \in \Omega, \, k \in \Lambda}$ solves the matrix equation $B(F,G,\Lambda,\Omega)C = B(F,G,\Lambda,\Lambda)$.  Then, $$(\la f,g_j \ra)_{j \in \Lambda}  = C^T ((\alpha_j)_{j \in \Omega} - (\beta_j)_{j \in \Omega}) + (\beta_j)_{j \in \Lambda},$$ where $C^T$ denotes the transpose of C.
\end{theorem}
\begin{proof}
Let $\{ f_j,g_j \}_{j \in \J}$ be a dual frame pair, $\Lambda$ be an erasure set, and $\Omega$ be a corresponding robust bridge set.  For $j \in \Lambda$ and $f \in \cH$
\begin{eqnarray*}
\la f,g_j \ra &=& \la f,g_j ' \ra + \la f,g_j - g_j ' \ra \\
&=& \la f,g_j ' \ra + \la f - f_R,g_j - g_j ' \ra + \la f_R,g_j - g_j ' \ra.
\end{eqnarray*}
Since $f - f_R \in \spn\{ f_j : j \in \Lambda \}$, equation (\ref{perpcond}) says that $f - f_R \perp g_j - g_j '$.  So,
\begin{eqnarray*}
\la f,g_j \ra &=& \la f,g_j ' \ra + \la f_R,g_j - g_j ' \ra \\
&=& \la f - f_R,g_j ' \ra + \la f_R,g_j \ra \\
&=& \sum_{k \in \Omega} \overline{c_k^{(j)}} \la f - f_R,g_k \ra + \la f_R, g_j \ra.
\end{eqnarray*}
Therefore, we can recover the erased coefficients with the following equation:
\[ (\la f,g_j \ra)_{j \in \Lambda} = C^T (\la f - f_R,g_k \ra)_{k \in \Omega} + (\la f_R,g_j \ra)_{j \in \Lambda}. \]
That is,
\[ (\la f,g_j \ra)_{j \in \Lambda}  = C^T ((\alpha_j)_{j \in \Omega} - (\beta_j)_{j \in \Omega}) + (\beta_j)_{j \in \Lambda}. \]
\end{proof}

\begin{example}\label{one erasure}
Consider the case where $\Lambda = \{ k \}$, and choose a set $\Omega = \{ \ell \}$.  Then, $g_k ' = c \, g_\ell$.  For Nilpotent bridging, we require that $\la f_k,g_k - g_k ' \ra = 0$.  In solving for c, we get
\[ 0 = \la f_k,g_k - g_k ' \ra = \la f_k,g_k \ra - \overline{c} \la f_k,g_\ell \ra. \]
So, if $\la f_k,g_\ell \ra \not = 0$, then $\Omega$ is a robust bridge set for $\Lambda$ and
\[ g_k ' = \frac{\la g_k,f_k \ra}{\la g_\ell,f_k \ra} g_\ell. \]
In particular any singleton set $\{ \ell \}$ is a robust bridge set for $\Lambda$ provided $\la f_k,g_\ell \ra \not = 0$.  So, in a suitably random frame, any singleton set disjoint from $\Lambda$ will be a robust bridge set. \qed
\end{example}

The following result provides a necessary and sufficient condition for the existence of a robust bridge set for a given erasure set.

\begin{theorem}\label{2-Nilpotent Bridging Theorem}
Let $(F,G)$ be a dual frame pair, and let $\Lambda$ be an erasure set.  Then there is a robust bridge set $\Omega$ for $\Lambda$ if and only if $\Lambda$ satisfies the minimal redundancy condition for $G$.  In this case we can take $|\Omega| = \dim(\mathcal{F})$, where $\mathcal{F} = \spn \{ f_j : j \in \Lambda \})$.
\end{theorem}
\begin{proof}
Assume that $\Lambda$ satisfies the minimal redundancy condition.  Let $\mathcal{F} = \spn \{ f_j : j \in \Lambda \}$. Let $q = \dim(\mathcal{F})$.  Let $\{ h_j \}_{j\in \N}$ be a basis for $\mathcal{F}^\perp$.  Since $\mathcal{F}^\perp$ has codimension $q$, we can complete this set to a basis $\{ h_j \}_{j \in \N} \cup \{ g_{j_k} \}_{k=1}^q$, where each $j_k \in \Lambda^c$.  Let $\Omega = \{ j_k \}_{k=1}^q$. Then $|\Omega| = q$ and $\Lambda \cap \Omega = \emptyset$.  For each $\ell \in \Lambda$, write
\[ g_\ell = \sum_{k=1}^q c_{j_k}^{(\ell)} g_{j_k} + \sum_{j \in \J} b_j^{(\ell)} h_j. \]
Let
\[ g_\ell ' = \sum_{k=1}^q c_{j_k}^{(\ell)} g_{j_k}. \]
Then $g_\ell - g_\ell ' \in \mathcal{F}^\perp$.  Therefore, by (\ref{perpcond}), the $c_{j_k}^{(\ell)}$ solve the bridge equation (\ref{BE}) and $\Omega$ is a robust bridge set.

To prove the converse, assume that $\Omega$ is a robust bridge set.   Assume that $f \perp \cspn\{g_j : j \in \Lambda^c\}$.  Then,
\[ f = \sum_{j \in \J} \la f,g_j \ra f_j = \sum_{j \in \Lambda} \la f,g_j \ra f_j. \]
So, $f \in \spn\{ f_j : j \in \Lambda \}$.  We have
\[ f = \sum_{j \in \Lambda} \la f,g_j - g_j' \ra f_j + \sum_{j \in \Lambda} \la f,g_j' \ra f_j. \]
However, since $f \in \spn \{ f_j : j \in \Lambda \}$, equation (\ref{perpcond}) says that $\la f,g_j - g_j' \ra = 0$ for all $j \in \Lambda$.  Since $g_j' \in \spn\{ g_j : j \in \Lambda^c \}$, $\la f,g_j' \ra = 0$ for all $j \in \Lambda$.  Hence, $f = 0$.  Therefore, $\cH = \cspn\{g_j : j \in \Lambda^c\}$ and $\Lambda$ satisfies the minimal redundancy condition with respect to $G$.
\end{proof}

The following is a useful criterion for sufficiency of robustness of a bridge set.

\begin{theorem}\label{cond1}
Let $(F,G)$ be a dual frame pair, and $\Lambda$ be an erasure set.  If $\Omega \subset \Lambda^c$ is a bridge set for which
\begin{equation}\label{rankcond}
\rank(B(F,G,\Lambda,\Omega)) = \dim(\cF)
\end{equation}
where $\cF = \spn\{ f_j : j \in \Lambda \}$, then $\Omega$ is a robust bridge set.  In particular if $|\Lambda|=|\Omega|$ and $B(F,G,\Lambda,\Omega)$ is invertible, then $\Omega$ is a robust bridge set.
\end{theorem}
\begin{proof}
First consider the special case where $\{ f_j : j \in \Lambda \}$ is a linearly independent set and $|\Lambda| = |\Omega|$.  The rank condition (\ref{rankcond}) is then just the condition that $B(F,G,\Lambda,\Omega)$ is invertible.  Then the system (\ref{BE}) has a unique solution $C = B(F,G,\Lambda,\Omega)^{-1}B(F,G,\Lambda,\Lambda)$.  So, $\Omega$ is robust.

Now consider the general case.  Let $\kappa = \text{rank}(B(F,G,\Lambda,\Omega)) = \dim{\cF}$.  Let $\Lambda_0 \subset \Lambda$ be such that $\{ f_j : j \in \Lambda_0 \}$ is a basis for $\cF$.  The rows of $B(F,G,\Lambda,\Omega)$ are linear combinations of the rows of $B(F,G,\Lambda_0,\Omega)$.  Thus, $\text{rank}(B(F,G,\Lambda_0,\Omega)) = \kappa$.  So, $|\Lambda_0| = \kappa$.  Then, there is a subset $\Omega_0 \subset \Omega$ with $|\Omega_0| = \kappa$ so that $\text{rank}(B(F,G,\Lambda_0,\Omega_0)) = \kappa$.  By the first paragraph of this proof, $\Omega_0$ is a robust bridge set for $\Lambda_0$.  The rows of $B(F,G,\Lambda,\Omega_0)$ are linear combinations of the rows of $B(F,G,\Lambda_0,\Omega_0)$ and the rows of $B(F,G,\Lambda,\Omega_0)$ are the {\it same} linear combinations of the rows of $B(F,G,\Lambda_0,\Omega_0)$.  It follows that $\Omega_0$ is a robust bridge set for $\Lambda$.  So since $\Omega$ contains $\Omega_0$, $\Omega$ is a robust bridge set for $\Lambda$.
\end{proof}

\begin{remark}
Theorem \ref{cond1} says that the rank condition (\ref{rankcond}) on the bridge matrix is sufficient for robustness of $\Omega$.  In the general case it is not necessary, as shown by Example \ref{nilpex}.  In that case, the unreduced error operator is already nilpotent of index 2, so any bridge set is robust for it.  From experiments, it appears that the minimal rank possible of the bridge matrix for a robust bridge set and the minimal size of $\Omega$ is linked to the number of nonzero eigenvalues of the unreduced error operator.  (See Theorem \ref{numev} for a result relating to this.)  However, for Parseval frames, the converse of Theorem \ref{cond1} holds.
\end{remark}

\begin{corollary}
Let $F$ be a Parseval frame.  If $\Lambda$ is an erasure set for $F$, and $\Omega \subset \Lambda^c$, then $\Omega$ is a robust bridge set for $\Lambda$ if and only if $\text{rank}(B(F,G,\Lambda,\Omega)) = \dim(\cF)$, where $\cF = \spn\{ f_j : j \in \Lambda \}$.  In particular, if $\{ f_j : j \in \Lambda \}$ is linearly independent and $|\Omega| = |\Lambda|$, then $\Omega$ is a robust bridge set for $\Lambda$ if and only if $B(F,G,\Lambda,\Omega)$ is invertible.
\end{corollary}
\begin{proof}
The ``only if'' part holds by Theorem \ref{cond1} for the dual frame pair $(F,G)$ with $G = F$.  For the ``if'' part, suppose $F$ is a Parseval frame, $\Lambda$ is an erasure set, and $\Omega$ is a robust bridge set for $\Lambda$.  By definition, for each $j \in \Lambda$ there exists $f_j ' \in \spn \{ f_k : k \in \Omega \}$ such that $f_j - f_j ' \in \cF^\perp$, where $\cF = \spn \{ f_j : j \in \Lambda \}$.  Let $\kappa = \dim(\cF)$.  Let $P$ be the orthogonal projection onto $\cF$.  Then for $j \in \Lambda$, $f_j = Pf_j = Pf_j'$.  So $\cF = P \spn \{ f_j ' : j \in \Omega \}$.  It follows that $|\Omega| \geq \kappa$.  Since $\spn \{ Pf_j : j \in \Omega \} = \cF$, there is a subset $\Omega_0 \subset \Omega$ such that $\{ Pf_j : j \in \Omega_0 \}$ is a basis for $\cF$.  Similarly there is a subset $\Lambda_0 \subset \Lambda$ such that $\{ f_j : j \in \Lambda_0 \}$ is a basis for $\cF$.  Then $\{ f_j : j \in \Lambda_0 \}$ and $\{ Pf_j : j \in \Omega_0 \}$ are two bases for the same Hilbert space, so $|\Lambda_
0| = |\Omega_0|$, and the cross Gramian $(\la Pf_k,f_j \ra)_{j \in \Lambda_0, \, k \in \Omega_0}$ is invertible (see the preliminaries), so it has rank $\kappa$.  But for each $j \in \Lambda_0$ and $k \in \Omega_0$,
$$\la Pf_k,f_j \ra = \la f_k,Pf_j \ra = \la f_k,f_j \ra,$$
so $(\la Pf_k,f_j \ra)_{j \in \Lambda_0, \, k \in \Omega_0}$ is just the bridge matrix $B(F,G,\Lambda_0,\Omega_0)$.  Since it has rank $\kappa$, and it is a minor of $B(F,G,\Lambda,\Omega)$, $\text{rank}(B(F,G,\Lambda,\Omega)) \geq \kappa = \dim(\cF)$.  But $\dim(\cF) = \kappa$ implies that $B(F,G,\Lambda,\Omega)$ can not have more than $\kappa$ linearly independent rows, and hence $\text{rank}(B(F,G,\Lambda,\Omega)) \leq \kappa$.  Thus the rank of the bridge matrix must be $\kappa$.  It follows that if $\{ f_j : j \in \Lambda \}$ are linearly independent then $B(F,G,\Lambda,\Omega)$ is invertible as claimed.
\end{proof}

The next two examples illustrate the relationship between the minimal redundancy condition and the invertability of $R_\Lambda$.  For the examples, we consider the dual frame pair
\[ F = \left\{ (1,1)^T, (-1,1)^T, (-1,-1)^T, (1,-1)^T \right\} \]
and
\[ G = \left\{ (1,0)^T, \left(\frac{1}{2}, \frac{1}{2}\right)^T, \left(\frac{1}{2}, -\frac{1}{2}\right)^T, (1,0)^T \right\}. \]

Our first example is an example where the 2-nilpotent bridging algorithm works, but $R_\Lambda$ is not invertible.
\begin{example}
Let $\Lambda = \{ 1 \}$,
\[ R_\Lambda = \sum_{j = 2}^4 f_j \otimes g_j = I - f_1 \otimes g_1 = \left(
                                                                        \begin{array}{cc}
                                                                          1 & 0 \\
                                                                          0 & 1 \\
                                                                        \end{array}
                                                                      \right) - \left(
                                                                                  \begin{array}{cc}
                                                                                    1 & 0 \\
                                                                                    1 & 0 \\
                                                                                  \end{array}
                                                                                \right) = \left(
                                                                                            \begin{array}{cc}
                                                                                              0 & 0 \\
                                                                                              -1 & 1 \\
                                                                                            \end{array}
                                                                                          \right)
 \]
is not invertible.  Therefore, methods that require the inversion of $R_\Lambda$ won't work.  Furthermore,
\[ E_\Lambda = \left(
                         \begin{array}{cc}
                           1 & 0 \\
                           1 & 0 \\
                         \end{array}
                       \right)
 \]
is idempotent, so Neumann series approximations also fail.  However, since $\la f_1,g_2 \ra \not = 0$ and $\la f_1,g_4 \ra \not = 0$, example \ref{one erasure} shows that nilpotent bridging works with $\Omega = \{ 2 \}$ or $\Omega = \{ 4 \}$.  Note that $\Omega = \{3 \}$ won't work for Nilpotent bridging since $\la f_1,g_3 \ra = 0$. \qed
\end{example}

While for robustness $\Lambda$ needs to satisfy the minimal redundancy condition with respect to $G$, the second example shows that $\Lambda$ need not satisfy the minimal redundancy condition with respect to $F$.

\begin{example}
Let $\Lambda = \{ 2,4 \}$, and $\Omega = \{ 1,3 \}$.  Then $\Lambda$ does not satisfy the minimal redundancy condition for $F$.  But, we have
\begin{eqnarray*}
f_2,f_4 &\perp& g_2 - 0g_1 - 0g_3 \ \ \text{and} \\
f_2,f_4 &\perp& g_4 - g_1 - 0 g_3.
\end{eqnarray*}
Letting $f = (4,2)^T$, we get
\[ f_R = R_\Lambda f = (f_1 \otimes g_1)(f) + (f_3 \otimes g_3)(f) = (3,3)^T \]
and
\[ f_B = B_\Lambda f = (f_2 \otimes 0)(f) + (f_4 \otimes g_1)(f) = (4,-4)^T. \]
So,
\[ \tilde{f} = f_R + f_B = (7,-1)^T. \]
We have
\[ f_{\tilde{E}} = \tilde{E}_\Lambda f_R = (f_2 \otimes (g_2 - 0g_1 - 0g_3))(f_R) + (f_4 \otimes (g_4 - g_1 - 0g_3))(f_R) = (-3,3)^T. \]
Therefore we recover our original vector as
\[ \tilde{f} + f_{\tilde{E}} = (4,2)^T. \] \qed
\end{example}

Consider a dual frame pair $(F,G)$ with erasure set $\Lambda$ and bridge set $\Omega$.  Computer experiments indicated  that if $|\Omega| < |\Lambda|$, then $|\sigma(\tilde{E}_\Lambda) \setminus \{ 0 \}| = |\Lambda| - |\Omega|$.  So, if one chooses a bridge set that is too small, $\tilde{E}_\Lambda$ will have nonzero eigenvalues, but may have fewer nonzero eigenvalues than $E_\Lambda$ (the error operator without bridging).  The following gives a mathematical proof of this fact.

\begin{theorem}\label{numev}
Let $(F,G)$ be a dual frame pair.  Assume $\Lambda$ satisfies the minimal redundancy condition with respect to $G$, and $|\Lambda| = L$.  Then, there is a bridge set $\Omega$ of any size $M \leq L$ so that $|\sigma(\tilde{E}_\Lambda) \setminus\{ 0 \}| \leq L - M$.
\end{theorem}
\begin{proof}
By Theorem \ref{2-Nilpotent Bridging Theorem}, we can find a robust bridge set $\Omega ' \subset \Lambda^c$ satisfying $|\Omega'| < L$.  That is, for each $k \in \Lambda$ we can find
\[ g_k ' = \sum_{j \in \Omega '} c_j^{(k)} g_j \]
so that $g_k ' \perp \spn\{ f_j : j \in \Lambda \}$.  Assume that $\Omega ' = \{ \omega_1, \cdots, \omega_{|\Omega '|} \}$.  Let $\Omega = \{ \omega_1, \cdots, \omega_M \}$ and
\[ g_k '' = \sum_{j \in \Omega} c_j^{(k)} g_j. \]
Then,
\[ \tilde{E}_\Lambda = \sum_{k \in \Lambda} f_k \otimes (g_k - g_k '') = \sum_{k \in \Lambda} f_k \otimes (g_k - g_k ') + \sum_{k \in \Lambda} f_k \otimes (g_k ' - g_k ''). \]
Let $N = \tilde{E}_\Lambda = \sum_{k \in \Lambda} f_k \otimes (g_k - g_k ')$, and $A = \sum_{k \in \Lambda} f_k \otimes (g_k ' - g_k '')$.  Then, it is easily verified that $N$ is nilpotent of index 2, and $NA = 0$.  Since $\text{range}(A^*) \subset \{ g_k ' - g_k '' : k \in \Lambda \} \subset \{ g_{\omega_k} : k = M+1, \cdots, |\Omega'| \}$, the rank of $A$ is at most $L-M$.

Let $\lambda \in \sigma(N+A) \setminus \{ 0 \}$.  Both $N$ and $A$ are finite rank operators, so $\lambda$ must be an eigenvalue of $N+A$.  Thus, there exists $x \in \cH$ so that
\[ (N + A) x = \lambda x. \]
Multiplying by $N$ on the left on both sides yields
\[ 0 = \lambda Nx. \]
Since $\lambda \not = 0$, we have $Nx = 0$.  Thus, $Ax = \lambda x$ and $\lambda \in \sigma(A)$.  Since $A$ can have at most $L - M$ distinct eigenvalues, it follows that $\tilde{E}_\Lambda$ has at most $L - M$ nonzero eigenvalues.
\end{proof}

\section{Applications to Sampling Theory}

There are well-known deep established connections between frame theory and modern sampling theory.  We cite for instance the excellent references (\cite{BF}, \cite{G}, \cite{Z}).    We note that a good account of sampling theory for our purposes is contained in Chapter 9 of \cite{HKLW}.  Let $X$ be a metric space and let $\mu$ be a Borel measure on $X$.  Let $\cH$ be a closed subspace of $L^2(X,\mu)$ consisting of continuous functions.  Let $T = \{ t_j \}_{j \in \J} \subset X$ and define the sampling transfrom $\Theta$ mapping $\cH$ into the complex sequences by $\Theta(f) = (f(t_j))_{j \in \J}$.  If $\Theta: \cH \rightarrow \ell^2(\J)$ and is bounded, then the point evaluation functionals $\gamma_j : \cH \rightarrow \C$ defined by $\gamma_j(f) = f(t_j)$ are bounded, and hence by the Riesz Representation Theorem, $\gamma_j(f) = \la f,g_j \ra$ for some $g_j \in \cH$.  If the sampling transform is also bounded below, then $\{ g_j \}_{j \in \J}$ forms a frame for $\cH$, and thus we can find some dual $\{
f_j \}_{j \in \cH}$.  We then have the identity
\begin{equation}
f = \sum_{j \in \J} \la f,g_j \ra f_j = \sum_{j \in \J} f(t_j) f_j \ \ \ \forall f \in \cH.
\end{equation}
We will refer to $(X,F,T)$ as a sampling scheme for $\cH$.  The most well known sampling scheme comes from the Shannon-Whittaker Sampling Theorem.  For this scheme, $\cH = PW[-\pi,\pi]$, $T = p \Z$ ($p \in (0,1]$), $f_j = \sinc(\pi(t - jp))$. Then $g_j = p \, \sinc(\pi(t - jp))$,  where $\sinc(x) = \frac{\sin x}{x}$.

Let $\Lambda$ be an erasure set for a sampling scheme $(X,F,T)$, with corresponding bridge set $\Omega$.  We can think of the erased coefficients as either $\la f,g_j \ra$ or as $f(t_j)$ for $j \in \Lambda$.  For this case, the bridge matrix is
\begin{equation}
B(F,G,\Lambda,\Omega) = \left( \la f_j,g_k \ra \right)_{j \in \Lambda, \, k \in \Omega} = \left( f_j(t_k) \right)_{j \in \Lambda, \, k \in \Omega}.
\end{equation}
Similarly, $B(F,G,\Lambda,\Lambda) = (f_j(t_k))_{j,k \in \Lambda}$.  Note that these matrices only involve the sampled values of the $\{f_j\}$ over the points $\{t_k\}$ and do not explicitly involve the $\{g_k\}$. Let us simply write $B(\Lambda,\Omega)$ and $B(\Lambda,\Lambda)$ for these two matrices. Then Theorem \ref{algorithm} becomes the following Theorem:

\begin{theorem}
Let $(X,F,T)$ be a sampling scheme with erasure set $\Lambda$ satisfying the minimal redundancy condition, and $\Omega$ be a robust bridge set for $\Lambda$.  Suppose $C = \left( \overline{c_j^{(k)}} \right)_{j \in \Omega, \, k \in \Lambda}$ solves the bridging equation $$B(\Lambda,\Omega) C = B(\Lambda,\Lambda),$$ where   $B(\Lambda,\Omega) =  \left( f_j(t_k) \right)_{j \in \Lambda, \, k \in \Omega}$ and $B(\Lambda,\Lambda) = (f_j(t_k))_{j,k \in \Lambda}$.
Then $$(f(t_j))_{j \in \Lambda} = C^T ((f(t_j))_{j \in \Omega} - (f_R(t_j))_{j \in \Omega}) + (f_R(t_j))_{j \in \Lambda}.$$

\end{theorem}

\section{Generic Duals}

In this section we deal only with finite frames in finite dimensional Hilbert spaces.  Assume that $\cH$ is an $n$-dimensional Hilbert space.  We denote the set of $N$-tuples of vectors in $\cH$ by $\cH^N$.  The space $\cH^N$ can be equipped with many equivalent norms, but the one we will use is defined by $\lb F \rb := \max_{1 \leq j \leq N} \lb f_j \rb$ for $F = \{ f_j \}_{j=1}^N \in \cH^N$.  Let $F = \{ f_j \}_{j=1}^N$ be a frame in $\cH^N$.  For a frame, $F$, we define $F^* = \{ G \in \cH^N : (F,G) \text{ is a dual frame pair} \}$ and call it the {\it dual set} of $F$.

In the frame literature, a class of frames is sometimes called generic if it is open and dense in the set of all frames (c.f. \cite{ACM}, \cite{BCE}, \cite{LD}).  We will say that a class of duals to a given frame $F$ is generic if it is open and dense in the relative topology on $F^*$ inherited as a subspace of $\cH^N$.

\begin{proposition}
$F^*$ is a closed, convex subset of $\cH^N$.
\end{proposition}
\begin{proof}
Let $G,G' \in F^*$.  Then, for any $t \in [0,1]$, we see that
\[ \sum_{j=1}^N f_j \otimes \left((1-t)g_j + tg_j ' \right) = (1-t) \sum_{j=1}^N f_j \otimes g_j + t \sum_{j=1}^N f_j \otimes g_j ' = (1-t)I+tI = I. \]
Hence, $(1-t)G + tG' \in F^*$, so $F$ is convex.  The proof that $F^*$ is closed is elementary.
\end{proof}

Since $F^*$ is a closed subset of $\cH$, $F^*$ is a complete metric space with the norm topology inherited from $\cH^N$.

\begin{theorem}
Let $\Lambda$ be an erasure set for a frame $F$ with the minimal redundancy condition and let $\{ g_j \}_{j \in \Lambda}$ be assigned arbitrarily.  Then, $\{ g_j \}_{j \in \Lambda}$ can be extended to a dual frame $\{ g_j \}_{j=1}^N \in F^*$.
\end{theorem}
\begin{proof}
We first show that under the same conditions on $F$, the set $\{ h_j \}_{j \in \Lambda}$ can be extended to $\{ h_j \}_{j=1}^N$ so that $\sum_{j=1}^N f_j \otimes h_j = 0$.  Let $A = \sum_{j \in \Lambda} f_j \otimes h_j$.  Let $\{ k_j \}_{j \in \Lambda^c}$ be a dual to the reduced frame $\{ f_j \}_{j \in \Lambda^c}$.  Then, $I = \sum_{j \in \Lambda^c} f_j \otimes k_j$.  So,
\[ A = \left( \sum_{j \in \Lambda^c} f_j \otimes k_j \right) A = \sum_{j \in \Lambda^c} f_j \otimes (A^* k_j). \]
For each $j \in \Lambda^c$, let $h_j = -A^* k_j$.  Then,
\[ \sum_{j=1}^N f_j \otimes h_j = \sum_{j \in \Lambda^c} f_j \otimes h_j + \sum_{j \in \Lambda} f_j \otimes h_j
= - \sum_{j \in \Lambda^c} f_j \otimes A^* k_j + A = A - \left( \sum_{j \in \Lambda^c} f_j \otimes k_j \right) A = A - IA = 0. \]

Now, let $\{ g_j ' \}_{j=1}^N \in F^*$.  Let $h_j = g_j - g_j '$ for $j \in \Lambda$.  Then, as above, we can extend $\{ h_j \}_{j \in \Lambda}$ to $\{ h_j \}_{j=1}^N$ so that $\sum_{j=1}^N f_j \otimes h_j = 0$.  For all $j$, let $\tilde{g_j} = g_j ' + h_j$.  So,
\[ \sum_{j =1}^N f_j \otimes \tilde{g_j} = \sum_{j=1}^N f_j \otimes g_j' + \sum_{j=1}^N f_j \otimes h_j = I + 0 = I. \]
Thus, $\{ \tilde{g}_j \}_{j=1}^N \in F^*$.  Furthermore, for $j \in \Lambda$,
\[ \tilde{g_j} = g_j ' + h_j = g_j ' + g_j - g_j ' = g_j. \]
Therefore, $\{ \tilde{g}_j \}_{j=1}^N$ is the desired extension of $\{ g_j \}_{j \in \Lambda}$.
\end{proof}

\begin{remark}
The above theorem shows that in the presence of the minimal redundancy condition, one can pick ``{\it designer duals}'' that satisfy certain conditions with respect to  $\Lambda$.  Theorem \ref{generic} (below) is our main result in this direction.
\end{remark}

In the frame literature (c.f. \cite{ACM}, \cite{CLTW}) a frame $F$ is said to have {\it spark $k$} if every collection of k vectors in $F$ is linearly independent, and $F$ is said to have the full spark property if it has spark $n$ (the dimension of $\cH$).  It is known that the set of full spark frames is an open dense set in $\cH^N$ (c.f. \cite{ACM}, \cite{LD}).

\begin{lemma}
Let $(F,G)$ be a dual frame pair with erasure set $\Lambda$, and bridge set $\Omega$ satisfying $|\Lambda| = |\Omega|$.  A necessary (but not sufficient) condition for $B(F,G,\Lambda,\Omega)$ to be an invertible matrix is
\begin{equation}
|\Lambda| \leq \min \left\{ n,N-n,\frac{N}{2} \right\}
\end{equation}
\end{lemma}
\begin{proof}
If $|\Lambda| > n$, then the rows of the bridge matrix $B(F,G,\Lambda,\Omega)$ will be linearly dependent (since $\cH$ is an n-dimensional space).  Thus, $B(F,G,\Lambda,\Omega)$ will fail to be invertible.

Assume that $B(F,G,\Lambda,\Omega)$ is invertible, and $|\Lambda| > N-n$.  Then, since the bridge equation $B(F,G,\Lambda,\Omega) C = B(F,G,\Lambda,\Lambda)$ has a solution ($C = B(F,G,\Lambda,\Omega)^{-1} B(F,G,\Lambda,\Lambda)$), Theorem $\ref{2-Nilpotent Bridging Theorem}$ asserts that $\Lambda$ satisfies the minimal redundancy condition with respect to $G$.  Therefore, $|\Lambda^c| \geq n$.  So, $N = |\Lambda| + |\Lambda^c| > N-n +n > N$.  This is a contradiction, and therefore, if $B(F,G,\Lambda,\Omega)$ is invertible, then $|\Lambda| \leq N-n$.

If $|\Lambda| > \frac{N}{2}$, then $|\Lambda| + |\Omega| > N$.  This is a contradiction since $\Lambda$ and $\Omega$ are disjoint subsets of $\{ 1, \cdots, N \}$.
\end{proof}

\begin{corollary}
Assume that $F \in \cH^N$ satisfies the full spark property.  Let $\Lambda$ be an erasure set satisfying $|\Lambda| \leq \min \{ n, N-n, \frac{N}{2} \}$, and $\Omega$ be a bridge set satisfying $|\Lambda| = |\Omega|$ and $\Lambda \cap \Omega = \emptyset$.  Then there exists a dual frame $G$ to $F$ so that $B(F,G,\Lambda,\Omega)$ is invertible.
\end{corollary}
\begin{proof}
To prove the corollary, define a bijection $\varphi: \Omega \rightarrow \Lambda$.  Let $\{ g_j \}_{j \in \Omega} = \{ f_{\varphi(j)} \}_{j \in \Omega}$.  By the previous lemma, we can extend $\{ g_j \}_{j \in \Omega}$ to a dual frame $G$ for $F$.  Then  $B(F,G,\Lambda,\Omega)$ is identical to the Gram matrix of the finite sequence $\{f_j : j \in \Lambda \}$, which is invertible since $\{f_j : j \in \Lambda \}$ is linearly independent.
\end{proof}

We say that a dual frame pair $(F,G)$ has {\it skew spark $k$} if for every erasure set $\Lambda$ with $|\Lambda| \leq k$, and any bridge set $\Omega \subset \Lambda^c$ satisfying $|\Lambda|=|\Omega|$, $B(F,G,\Lambda,\Omega)$ is invertible.  If $(F,G)$ has skew spark $\min\{ \frac{N}{2},n,N-n \}$, then $(F,G)$ is said to satisfy the full skew spark property.

\begin{proposition}
If the dual frame pair $(F,G)$ for $\cH$ has skew spark $k$, then $F$ and $G$ each have spark $k$.
\end{proposition}
\begin{proof}
Let $\Lambda$ be an erasure set of cardinality $k$.  Let $\Omega$ be any subset of $\Lambda^c$ of cardinality $k$.  By hypothesis the matrix $B(F,G,\Lambda,\Omega)$ is invertible, so its rows and columns are linearly independent.  This implies that $\{f_j: j \in \Lambda\}$ is linearly independent.  Since $\Lambda$ was arbitrary, this shows that $F$ has spark $k$.  The proof for $G$ is analogous.

\end{proof}

This also shows that if $n \leq \min\{ \frac{N}{2},N-n \}$, then the full skew spark property implies the full spark property.

Let $\mathcal{G} = \{ G \in F^* : (F,G) \text{ has the full skew spark property} \}$.

\begin{theorem}\label{generic}
Assume that $F$ has the full spark property.  Then, $\mathcal{G} := \{ G \in F^* : (F,G) \text{ has the full skew spark property} \}$ is generic in $F^*$.
\end{theorem}
\begin{proof}
Let $\Gamma = \{ \Lambda \subset \{ 1,\cdots,N \} : |\Lambda| \leq \min\{ \frac{N}{2},N-n,n \} \}$.  For a given $\Lambda \in \Gamma$, let $\Phi_\Lambda = \{ \Omega \subset \{ 1, \cdots,N \} : |\Omega| = |\Lambda|,\Omega \cap \Lambda = \emptyset \}$.  Then, $\mathcal{G} = \bigcap_{\Lambda \in \Gamma} \bigcap_{\Omega \in \Phi_\Lambda} \mathcal{G}_{\Lambda,\Omega}$, where $\mathcal{G}_{\Lambda, \Omega} = \{ G \in F^* : \det(B(F,G,\Lambda,\Omega)) \not = 0 \}$.  Since we are intersecting over all possible erasure sets and all corresponding bridge sets, the above intersection is finite.  So by the Baire category theorem, if we show that each $\mathcal{G}_{\Lambda, \Omega}$ is open and dense, then $\mathcal{G}$ will also be open and dense.

Fix an erasure set $\Lambda$, and a corresponding bridge set $\Omega$.  It is easily verified that the maps $G \xmapsto{\alpha} B(F,G,\Lambda,\Omega)$ and $B(F,G,\Lambda,\Omega) \mapsto \det(B(F,G,\Lambda,\Omega))$ are continuous.  So, $\mathcal{G}_{\Lambda, \Omega} = (\det \circ \alpha)^{-1} (\C \setminus \{0\})$ is an open set.

To show density of $\mathcal{G}_{\Lambda, \Omega}$, let $\epsilon > 0$, and assume that $G_0 \in F^* \setminus \mathcal{G}_{\Lambda, \Omega}$.  Since $F$ satisfies the full spark property, $\Lambda$ satisfies the minimal redundancy condition with respect to $F$.  Thus, by Corollary 5.5, there is a $G_1 \in F^*$ so that $\det(B(F,G_1,\Lambda,\Omega)) \neq 0$ .  Let $G_t = (1-t)G_0 + t G_1$.  By proposition 6.1, $G_t \in F^*$.  Furthermore, $\det(B(F,G_t,\Lambda,\Omega))$ is a polynomial in $t$ satisfying $\det(B(F,G_t,\Lambda,\Omega))(0) = 0$ and $\det(B(F,G_t,\Lambda,\Omega))(1) \not = 0$.  Thus, $\det(B(F,G_t,\Lambda,\Omega))$ has only finitely many zeros.  So, we can find $0 < t_0 < \frac{\epsilon}{\lb G_1 - G_0 \rb}$ so that $G_{t_0} \in \mathcal{G}_{\Lambda, \Omega}$.  Furthermore,
\[ \lb G_{t_0} - G_0 \rb = \lb (1-t_0)G_0 + t_0 G_1 - G_0 \rb = \lb t_0 (G_1 - G_0) \rb \leq |t_0| \lb G_1 - G_0 \rb < \epsilon. \]
Hence, $\mathcal{G}_{\Lambda, \Omega}$ is dense in $F^*$.

Therefore, by the Baire-Category theorem, $\mathcal{G}$ is generic in $F^*$.
\end{proof}

In short, what we have proven in this section is that for most frames $F \in \cH^N$, and most duals $G$ to $F$, the pair $(F,G)$ satisfies the full skew spark property.

\begin{remark}

We found it convenient to present and prove the topological resuts of this section for the metric topology.  A similar argument can be used to obtain these for the Zariski topology.

\end{remark}

\section{Computing an Inverse for $R_\Lambda$}

In this section, we obtain a basis-free closed-form formula for the inverse of the partial reconstruction operator $R_\Lambda$ for a finite erasure set. By basis-free we mean that the computations do not depend on any preassigned basis for the space, and by closed-form we mean that it is of the same general form as $R_\Lambda$ is given in and does not require an iterative process such as the Neuman series  formula.  This gives a second method of perfect reconstruction from frame and sampling erasures in finitely many steps that applies when $R_\Lambda^{-1}$ exists.

Let $(F,G)$ be a dual frame pair indexed by $\J$, and $\Lambda$ be an erasure set.  Recall that
$$R_\Lambda = \displaystyle \sum_{j \in \J \setminus \Lambda} f_j \otimes g_j = I - \displaystyle \sum_{j \in \Lambda} f_j \otimes g_j.$$
We derive a simple method for computing inverses of operators of the form
$$R  = I - \displaystyle \sum_{j = 1}^L f_j \otimes g_j$$
that was motivated by our work on bridging.

\begin{proposition}
Assume that $R  = I - \sum_{j = 1}^L f_j \otimes g_j$ is invertible.  Then, $R^{-1}$  has the form $\sum_{j,k = 1}^L c_{jk} f_j \otimes g_k$ for some $c_{jk} \in \C$.
\end{proposition}
\begin{proof}

We must show that $R^{-1} - I$ is a linear combination of the elementary tensors
$\{f_j \otimes g_k\}_{j,k = 1}^L$.

Let $E = \sum_{j=1}^L f_j \otimes g_j$.  Note that $(I - R^{-1}) R = R - I = - E$.  Since $R$ is invertible, this shows that the range of $I - R^{-1}$ is contained in $\spn \{ f_j \}_{j=1}^L$.

 We have,
\[ -E = R - I = (I-R^{-1}) R = (I - R^{-1})(I-E) = I - R^{-1} - (I-R^{-1})E. \]
Therefore,
\[ R^{-1} = I + E - (I-R^{-1})E. \]
From above, we know that $(I-R^{-1})f_k = \sum_{j=1}^L b_{jk} f_j$ for some $b_{jk} \in \C$.  So,
\[ (I-R^{-1})E = \sum_{k=1}^L (I-R^{-1}) f_k \otimes g_k = \sum_{k=1}^L \sum_{j=1}^L b_{jk} f_j \otimes g_k \]
a linear combination of the $\{f_j \otimes g_k\}_{j,k = 1}^L$ .

Since $E$ is also a linear combination of the $f_j \otimes g_k$, $E - (I - R^{-1})E = \sum_{j,k = 1}^L c_{jk} f_j \otimes g_k$ for appropriate constants  $c_{jk} \in \C$.

So $R^{-1} = I + \sum_{j,k = 1}^L c_{jk} f_j \otimes g_k$.
\end{proof}

Although the elementary tensors $f_j \otimes g_k$ in the representation of $R^{-1}$ in Proposition 6.1 are generally not linearly independent and hence the coefficients $\{c_{jk}\}_{j,k = 1}^L$ are not unique, we can derive a simple matricial formula that gives a valid choice of the $c_{jk}$.

\begin{theorem}
Let $R = I - \sum_{j = 1}^L f_j \otimes g_j$, where $\{f_j\}_{j=1}^L , \{g_j\}_{j=1}^L$ are finite sequences and $\{f_j\}_{j=1}^L$ is linearly independent.  If $R$ is invertible, then a formula for the inverse is
\begin{equation}
R^{-1} = I + \sum_{j,k=1}^L c_{jk} f_j \otimes g_k
\end{equation}
where the coefficient matrix $C := (c_{jk})_{j,k=1}^L$ is given by
\begin{equation}
C = (I - M)^{-1}
\end{equation}
where $I$ is the $L \times L$ identity matrix and
\begin{equation}
M = G(\{f_1,...,f_L\}, \{g_1,...,g_L\}) :=
\left( \begin{array}{cccc}
 \la f_1,g_1 \ra & \la f_2,g_1 \ra & \cdots & \la f_L,g_1 \ra\\
\la f_1,g_2 \ra & \la f_2,g_2 \ra & \cdots & \la f_L,g_2 \ra\\
\vdots & \vdots & \ddots & \vdots\\
\la f_1,g_L \ra & \la f_2,g_L \ra & \cdots & \la f_L,g_L \ra
                                       \end{array} \right)
\end{equation}

\end{theorem}

\begin{proof}

By Proposition 6.1 we can write \\
$R^{-1} = I + \displaystyle \sum_{j = 1}^L \displaystyle \sum_{k = 1}^L c_{jk} f_j \otimes g_k$ for some $c_{jk} \in \C$. Compute:
\begin{eqnarray*}
I & = & R^{-1} R \\
& = & \left( I + \sum_{j = 1}^L \sum_{k = 1}^L c_{jk} f_j \otimes g_k \right) \left( I - \sum_{j = 1}^L f_j \otimes g_j \right) \\
& = & I + \sum_{j = 1}^L \sum_{k = 1}^L c_{jk} f_j \otimes g_k - \sum_{j = 1}^L f_j \otimes g_j - \sum_{j = 1}^L  \sum_{k = 1}^L \sum_{\ell = 1}^L c_{jk} (f_j \otimes g_k)(f_\ell \otimes g_\ell) \\
& = & I + \sum_{j = 1}^L \sum_{k = 1}^L c_{jk} f_j \otimes g_k - \sum_{j = 1}^L f_j \otimes g_j - \sum_{j = 1}^L  \sum_{\ell = 1}^L \sum_{k = 1}^L c_{jk} \la f_\ell,g_k \ra(f_j \otimes g_\ell) \\
& = & I + \sum_{j = 1}^L \sum_{k = 1}^L c_{jk} f_j \otimes g_k - \sum_{j = 1}^L f_j \otimes g_j - \sum_{\ell = 1}^L  \sum_{j = 1}^L \sum_{k = 1}^L c_{j\ell} \la f_k,g_\ell \ra(f_j \otimes g_k).
\end{eqnarray*}
In the last sum, we switched indices $k$ and $\ell$.  Thus,
\begin{equation*}
\sum_{j = 1}^L f_j \otimes g_j = \sum_{j = 1}^L \sum_{k = 1}^L c_{jk} f_j \otimes g_k - \sum_{\ell = 1}^L  \sum_{j = 1}^L \sum_{k = 1}^L c_{j\ell} \la f_k,g_\ell \ra(f_j \otimes g_k).
\end{equation*}

By simply setting the coefficients of the $f_j \otimes g_k$ to zero, we obtain the following system of equations.
\begin{equation}
c_{jk} - \sum_{\ell = 1}^L c_{j\ell} \la f_k,g_\ell \ra = \delta_{jk}
\end{equation}

For a fixed value of $j$,
we have the system
\[ \left(\delta_{jk} \right)_{k = 1, \cdots, L}^T = \left( \begin{array}{cccc}
1 - \la f_1,g_1 \ra & \la f_1,g_2 \ra & \cdots & \la f_1,g_L \ra\\
\la f_2,g_1 \ra & 1 - \la f_2,g_2 \ra & \cdots & \la f_2,g_L \ra\\
\vdots & \vdots & \ddots & \vdots\\
\la f_L,g_1 \ra & \la f_L,g_2 \ra & \cdots & 1 - \la f_L,g_L \ra
                                       \end{array} \right) \left( c_{jk} \right)_{k = 1, \cdots, L}^T \]

Let $C = \left( c_{jk} \right)_{j,k}$. Combining the equations for all $j$ gives
$$ I = (I - M^T)C^T $$  where $M = G(\{f_1,...,f_L\}, \{g_1,...,g_L\})$.  So
$C(I - M) = I$.

We will show that under our hypothesis that $\{f_1, \cdots f_L\}$ is linearly independent the matrix $I-M$ is invertible, so this system has a unique solution.  This will yield a valid choice of the $c_{jk}$.  If $I - M$ were singular then $1$ would be an eigenvalue of $M$.
So there would exist a nonzero vector
$x = (x_k)_{k=1}^L \in \C^n$ so that $Mx = x$.  Computing gives $\sum_{j=1}^L \la f_j,g_k \ra x_j = x_k$ for each $k$.   Let $z = \sum_{j=1}^L x_j f_j$.  Since $x$ is nonzero not all of the $x_j$ are zero.  By hypothesis $\{f_1, \cdots f_L\}$ is linearly independent, so $z$ cannot be the zero vector.
Compute:
\[ Rz = z - \sum_{k=1}^L \la z,g_k \ra f_k = z - \sum_{k=1}^L \sum_{j=1}^L x_j \la f_j,g_k \ra f_k = z - \sum_{k=1}^L x_k f_k = z - z = 0. \]
So $z$ is in the kernel of $R$ contradicting our hypothesis that $R$ is invertible.  Thus $I-M$ is a nonsingular matrix, and the system has the unique solution $C = (I - M)^{-1}$ as claimed.

\end{proof}

\begin{remark}

In order to apply the above theorem to inverting a frame partial reconstruction operator, if $\{f_j: j \in \Lambda\}$ is not linearly independent one must first use linearity of the elementary tensors $f \otimes g$ in the first component and conjugate linearity in the second component to precondition $R$ to the form
$I - \sum_{j = 1}^L f'_j \otimes g'_j$ with the first component set $\{f'_j: j \in \Lambda\}$ linearly independent.  In many cases this will be simple and even automatic, but in other cases this may be computationally expensive.
The main point is that if $R$ is invertible, the computation above, perhaps with preconditioning, always yields a formula for the inverse.  Furthermore, it may be useful to note that since we are solving a matrix equation, it follows that the coefficients  $c_{jk}$ are given by rational functions of the $\la f_j,g_k \ra$. In this sense the formula is indeed basis-free.
\end{remark}

\section{Concluding Remarks}
The main results in this article are presented for reconstruction from finite erasure subsets of frames, so much of our theory is finite dimensional.  However, the reconstruction results can be applied to finite subsets of infinite frames, including the well known classes of Gabor (Weyl-Heisenberg) frames, Laurent frames, infinite group frames, and wavelet frames, as well as abstract sampling theory.  There may be applications to the pure and applied aspects of these classes, including classification results.  In fact, our initial computer experiments suggest to us that many of these natural classes of infinite frames may be full skew spark in the sense that they have skew spark k for all finite k.  But mathematical proofs of general theorems on this have eluded us so far. In addition, there may be applications to the three closely related topics that deal with frames in {\it blocks}: operator-valued frames, fusion frames, and G-frames (c.f.  \cite{KLZ}, \cite{CKL}, \cite{Su}).  Finally, we should mention
that we expect that there will be applications to the more abstract theories: frames for Banach spaces and related topics of Banach frames, atomic decompositions, and framings (c.f \cite{CHL}), the theory of frames for Hilbert C*-modules, and in the purely algebraic direction: frames for other fields such as p-adic frames and binary frames  (c.f. \cite{HLS}).


\bigskip
\bigskip


\begin{thebibliography}{A}

\bibitem[ACM]{ACM} B. Alexeev, J. Cahill, and D. Mixon, {\em Full Spark Frames}, J. Fourier Anal. Appl., vol. 18, is. 6 (2012) 1167-1194.

\bibitem[BCE]{BCE} R. Balan, P. G. Casazza, and D. Edidin, {\em Equivalence of Reconstruction from the Absolute Value
of the Frame Coefficients to a Sparse Representation Problem}, IEEE Signal Processing Letters 14 No. 5 (2007) 341-343.

\bibitem[BBCE]{BBCE} R. Balan, B. G. Bodmann, P. G. Casazza, and D. Edidin, {\em Frames for Linear Reconstruction
without phase}, The 42nd Annual Conference on Information Sciences and Systems (2008), 721-726.

\bibitem[BF]{BF} J. Benedetto and P.J.S.G. Ferriera (eds.), {\em Modern Sampling Theory}, Birkh\"{a}user, Boston, MA (2001).

\bibitem[BP]{BP} B.G. Bodmann and V.I. Paulsen, {\em Frames, Graphs and Erasures}, Linear Algebra Appl., vol 404 (2005), 118-146

\bibitem[BOG]{BOG} P. Boufounos, A. V. Oppenheim, and V. K. Goyal, {\em Causal Compensation for Erasures in Frame
Representations}, IEEE Transactions on Signal Processing vol. 56 no. 3 (2008) 1071-1082.

\bibitem[CHL]{CHL} P. G. Casazza, D. Han, and D. R. Larson,
{\em Frames for Banach spaces}, The functional and harmonic analysis of wavelets and frames (San Antonio, TX, 1999),
Contemp. Math. {\bf 247} (1999), 149--182.

\bibitem[CK]{CK} P. G. Casazza and J. Kova\v{c}evi\'{c}, {\em Equal-norm Tight Frames with Erasures}, Adv. Comp. Math., vol 18 (2003), 387-430.

\bibitem[CK2]{CK2} P. G. Casazza and G. Kutyniok, {\em Robustness of Fusion Frames under Erasures of Subspaces and of Local Frame Vectors}, Contemp. Math., 464 (2008) 149-160.

\bibitem[CKL]{CKL} P. G. Casazza, G. Kutyniok, and S. Li, {\em Fusion frames and distributed
processing}, Appl. Comput. Harmon. Anal. {\bf 25} (2008), 114--132.

\bibitem[CLTW]{CLTW} P. G. Casazza, R. G. Lynch, J. C. Tremain, and L. M. Woodland, {\em Integer Frames}, Houston Journal of Mathematics, to appear.

\bibitem[Ch]{Ch} O. Christensen, \emph{An introduction to frames and Riesz
bases}, Birkh\"{a}user, 2003.

\bibitem[DL]{DL}  X. Dai and D. R. Larson, \emph{Wandering vectors
for unitary systems and orthogonal wavelets}, Mem. Amer. Math. Soc.
{\bf 134}, 1998.

\bibitem[DS]{DS} R.J. Duffin and A.C. Schaeffer, {\emph A class of nonharmonic  Fourier series}, Trans. Amer. Math. Soc.,
 \textbf{72} (1952), 341--366.

\bibitem[GKK]{GKK} V. K. Goyal, J. Kova\v{c}evi\'{c}, and J. A. Kelner, {\em Quantized Frame Expansions with Erasures}, Appl. Comp. Harm. Anal., vol. 10 (2001), 203-233.

\bibitem[G]{G} K. Gr\"{o}chenig, {\em Foundations of Time Frequency Analysis}, Applied and Numerical Harmonic Analysis, Birkh\"{a}user Boston, Boston, MA (2001).

\bibitem[Han1]{Han1} D. Han \emph{Frame representations and parseval duals with applications to Gabor frames,}
 Trans. Amer. Math. Soc.,  {\bf 360} (2008), 3307--3326.

\bibitem[HP]{HP} R. Holmes and V.I. Paulsen, {\em Optimal Frames for Erasures}, Lin. Alg. Appl., vol. 377 (2004), 31-51.

\bibitem[HKLW]{HKLW} D. Han, K. Kornelson, D. Larson, and E. Weber.  {\em Frames for Undergraduates}.  American Mathematical Society Student Mathematical Library, vol. 40, AMS (2007).

\bibitem[HL]{HL}  D. Han and D. R. Larson, \emph{Frames, bases and group
representations}, Mem. Amer. Math. Soc. {\bf 697}, 2000.

\bibitem[HLS]{HLS} R. Hotovy, D. R. Larson, and S. Scholze, {\em Binary Frames}, Houston Journal of Mathematics, to appear.

\bibitem[KLZ]{KLZ} V. Kaftal, D. R. Larson, and S. Zhang, \emph{Operator valued frames},
Trans. Amer. Math. Soc., \textbf{361} (2009),6349--6385.

\bibitem[LD]{LD} Y. M. Lu and M. N. Do, {\em A Theory for Sampling Signals from a Union of Subspaces}, IEEE Transactions on Signal Processing vol. 56 no. 6 (2008) 2334-2345.

\bibitem[PHM]{PHM} S. Pehlivan, D. Han, and R. Mohapatra, {\em Linearly Connected Sequences and Spectrally Optimal
Dual Frames for Erasures}, J. Funct. Anal., to appear.

\bibitem[Su]{Su}   W. Sun, \emph{G-frames and g-Riesz bases}, J. Math.
Anal. Appl. {\bf 232} (2006), 437--452.

\bibitem[Z]{Z} A. I. Zayed, {\em Advances in Shannon's Sampling Theory}, CRC Press, Boca Raton, FL (1993).

\end{thebibliography}
\end{document}